\documentclass[10pt]{article}
\usepackage{amsmath,amsthm, amssymb,amsxtra}
\usepackage{graphicx}

\def\indeg{{\rm indeg}}

\def\a{{\mathfrak a}}

\def\P{{\mathbb P}}
\def\C{{\mathfrak C}}
\def\red{{\rm red}}
\def\CO{{\mathcal O}}

\def\antiddot{\mathinner{\mkern1mu\raise1pt\vbox{\kern7pt\hbox{.}}\mkern2mu
        \raise4pt\hbox{.}\mkern2mu\raise7pt\hbox{.}\mkern1mu}}

\newcommand{\ZZ}{{\mathbb Z}}

\newcommand{\ann}{{\rm{ann}}}
\newcommand{\Ext}{{\rm{Ext}}}

\newcommand{\s}{\mathcal}

\newcommand{\cO}{{\s O}}

\newcommand{\cI}{{\s I}}

\newcommand{\punkt}{\hspace{-.3ex}\raise.15ex\hbox to1ex{\Huge.}}

\def \fix#1 {{\hfill\break \bf (( #1 ))\hfill\break}}

\DeclareMathOperator{\reg}{reg}

\DeclareMathOperator{\Hom}{Hom}

\DeclareMathOperator{\Tor}{Tor}

\DeclareMathOperator{\codim}{codim}

\newtheorem{theorem}{Theorem}[section]
\newtheorem{lemma}[theorem]{Lemma}
\newtheorem{proposition}[theorem]{Proposition}
\newtheorem{corollary}[theorem]{Corollary}

\theoremstyle{definition}

\newtheorem{example}[theorem]{Example}

\title{The Regularity of the Conductor\
\footnote{The authors are grateful for the support of the National Science Foundation during the preparation of this work.}\\
Dedicated to Joe Harris, who has taught us so much, on the occasion of his 60th Birthday}
\author{David Eisenbud and Bernd Ulrich}

\date{}
\begin{document}

\maketitle

\begin{abstract}We bound the Castelnuovo-Mumford regularity and syzygies of the ideal of the singular set of a plane curve, and more generally of the conductor scheme of certain projectively Gorenstein varieties.
\end{abstract}

\section{Introduction}

This note was inspired by a letter from Remke Kloosterman asking whether
the following result (now essentially Proposition 3.6 in Kloosterman~\cite{Kloosterman}) was known:

\begin{theorem}\label{Kloosterman}
Suppose that $C\subset \P^{2}_{\bf C}$ is a reduced plane curve of degree $d$ over the complex numbers, and
suppose that the only singularities of $C$ are ordinary nodes and cusps, i.e., have local
analytic equations $xy=0$ or $y^2-x^3=0$. If
$\, \Gamma\subset \P^{2}_{\bf C}$ denotes the set of points at which $C$ has nodes, then
$\reg I_{\Gamma}\leq d-1$, and the minimal number of homogeneous generating syzygies
of  degree $d$ is precisely the number of irreducible components of $C$ minus 1.
\end{theorem}

In case $C$ is irreducible this result simply says that the regularity of the set of nodal points
of $C$ is bounded by $d-2$. Since the regularity of the set of nodal points is bounded by
the regularity of the set of all the singular points, this is a consequence of the classical ``completeness of the adjoint series'' (see Section~\ref{adjoint section}).

In the general case, Kloosterman's proof is based on delicate arguments of Dimca \cite{Dimca} about the mixed Hodge theory of singular hypersurfaces. Aside from the application to arithmetic geometry that Kloosterman makes, his result seemed to us interesting and important because it sheds some light on the famous problem of understanding the restrictions on the positions of the nodes of a plane curve, about which  little is known (see Section~\ref{other results}).

It is the purpose of this note to give a simple expression for the regularity of the conductor ideal that extends Kloosterman's result in a way not limited to characteristic zero or to curves; it is a statement about any finite birational extension of a quasi-Gorenstein ring by a Cohen-Macaulay ring. The proof, given in the next section, involves only considerations of duality.

Specializing to the case of plane curves, we can use the extra strength of our result
to extend Theorem~\ref{Kloosterman} to curves with arbitrary singularities (Corollaries~\ref{plane curves} and \ref{plane curves 2}). Here is a special case of that result:

\begin{corollary}\label{improved Kloosterman}
Suppose that $C\subset \P^{2}_K$ is a reduced plane curve of degree $d$ over an algebraically
closed field $K$.
If $\, \Gamma$ denotes any set of singular points of $C$, then
$\reg I_{\Gamma}\leq d-1$. If, moreover, $\, \Gamma$ contains the points $($if any$)$ at which distinct irreducible components of $C$ meet and all these points of intersection are ordinary nodes, then the minimal number of homogeneous generating syzygies of $I_{\Gamma}$
of degree $d$ is precisely the number of irreducible components of $C$ minus 1.
\end{corollary}

Theorem \ref{Kloosterman} follows because the set $\Gamma$ in the Theorem
contains the points at which distinct components of the curve meet.

Another ideal related to the singular set and the conductor is the Jacobian ideal of the plane curve. This
ideal is an ``almost complete intersection'' in characteristic zero. If the curve has only ordinary nodes as singularities, then the conductor ideal is the saturation of the Jacobian ideal. We prove a general result  (Proposition~\ref{reg from fitting}) about the syzygies of an almost complete intersection with perfect saturation that implies, in the situation of a degree $d$ plane curve $f=0$ with only nodes, that the syzygies of the partial derivatives of $f$ have degree at least $2d-3$, and a little more (Corollary~\ref{derivatives}); this generalizes a result of Dimca and Sticlaru~\cite{Dimca-Sticlaru} that, again, was originally proven by Hodge theory.

Besides the conductor one can measure the difference between a standard graded algebra $A$ and a (partial) normalization $B$ by the size of the $A$-module $B/A$. We also give bounds on the regularity
of this module (and on the regularity of $B$ as an $A$-module) in the case where both $A$ and
$B$ are Cohen-Macaulay and $A$ is reduced (Proposition~\ref{regB}).

The fact that the number of components of a plane curve appears in a formula for the regularity of the conductor suggests that there might be a simple relation between the conductor of a reducible hypersurface and the conductors of its components; such a relation is given in Proposition~\ref{reduced}.

\section{Notation and Conventions}
Throughout we let
$S = K[x_{0},\ldots, x_{r}]$ be a polynomial ring over a field $K$. If $M$ is a finitely generated
graded $S$-module, we write $\reg M$ for the (Castelnuovo-Mumford) regularity of $M$ and $\indeg\ M$ for the infimum of $\{i \mid M_{i}\neq 0\}$. Let $X$ be a subscheme of $\P^r_K$ with saturated homogeneous ideal
$I_{X}$ and homogeneous coordinate ring $A=S/I_X$. We define $\reg X :=\reg I_{X}= \reg A +1$. If $A \subset B$ is a ring extension, we denote by
\begin{align*}
\C_{B/A} &:= \ann_{A} ({B}/{A}) \subset A  \text{ and }\\
\C'_{B/A} &:= \ann_{S} ({B}/{A}) \subset S \, ,
\end{align*}
the \emph{conductor} of $B$ into $A$, regarded as an ideal
of $A$ or of $S$, respectively. Note that $A/\C_{B/A} = S/\C'_{B/A}$.

When $X$ is reduced and $B=\overline A$ is the normalization of $A$, we
write $\C_{X}$ or $\C'_{X}$ instead of $\C_{B/A}$ or $\C'_{B/A}$. These
are homogeneous ideals of codimension at least $1$ and $1+ \codim X$, respectively.

\section{Regularity}

Recall that a homogeneous ideal $I\subset S$ is called \emph{perfect} if $S/I \neq 0$ is a Cohen-Macaulay ring or, equivalently, if the projective dimension of the $S$-module $S/I$ is $\codim I$.

\begin{theorem}\label{main} Let $X\subset \P^{r}_{K}$ be a reduced scheme
of codimension $c$ with
homogeneous coordinate ring $A$, and assume that $A$ is Gorenstein. Let
$\overline A$ be the normalization of $A$.
Write $S = K[x_{0},\dots, x_{r}]$ for the homogeneous
coordinate ring of $\P^{r}_K$.
Suppose $B$ is a graded Cohen-Macaulay ring
with $A\varsubsetneq B\subset \overline A$.

The conductor ideal $\C'_{B/A} \subset S$ is perfect of codimension
$c+1$ and
$$
\reg \C'_{B/A} = \reg X  - 1 - \indeg (B/A) \, ;
$$
in particular, if $X$ is geometrically reduced and irreducible, then $\reg \C'_{B/A} < \reg X -1 \, .$

Moreover,
$
\Tor^{S}_{c}(\C'_{B/A}, K)_{\reg \C'_{B/A}+c}
$
is $K$-dual to
$(B/A)_{\indeg(B/A)};
$
in particular, the $c$-th syzygy module
of $\C'_{B/A}$ has precisely
$\dim_{K} (B/A)_{\indeg(B/A)}$ homogeneous minimal generators of highest degree.
\end{theorem}

Under mild hypotheses we can use Theorem~\ref{main} to extract information about the codimension 1 components of
the singular locus of $X$. The module $S/\C'_{X}$ is
supported precisely on these components. If the codimension 1 components are generically only nodes and cusps (that is, after we localize at such a component, complete, and extend the residue field to its algebraic closure, they become nodes and cusps), then
the conductor ideal is generically radical; and since the conductor ideal is perfect, it is equal to the reduced ideal
of the union of the codimension 1 components of the singular locus; thus this reduced scheme has
regularity $\leq \reg X -1$.

We remind the reader that regularity has a simple interpretation for perfect ideals:

\begin{proposition}\label{Cohen-Macaulay}
Suppose that $I\subset S$ is a homogeneous perfect ideal of codimension $c$ with
$1\leq c\leq r$, and let $\cI$ be the
corresponding ideal sheaf on $\P^{r}_K$. The following statements are equivalent$:$
\begin{itemize}
\item[$(1)$] $\reg S/I \leq m$;
\item[$(2)$] The $c$-th syzygies of $S/I$ are generated in degrees $\leq m+c$;
\item[$(3)$] $H^{r-c+1}(\cI(m+c -r )) = 0$;
\item[$(4)$] The value of the Hilbert function $\dim_{K}(S/I)_{e}$ is equal to the
corresponding value of the Hilbert polynomial of $S/I$ for all $e\geq m+c-r$.
\end{itemize}

Moreover, if $m=\reg S/I$, the number of highest syzygies of $S/I$ $($or $I$$)$ of highest degree can be computed in terms of cohomology via a natural isomorphism
$$
{\rm Tor}_{c}^S(S/I, K)_{m+c} \cong H^{r-c+1}(\cI(m+c-r-1)).
$$
\end{proposition}

\begin{proof} The equivalence of the first four statements is standard (see for example
Eisenbud [2005], Section 4). For the last statement, note that under the given
hypothesis the minimal homogeneous free
$S$-resolution of $S/I$ has length $c$, and the highest degree of a $c$-th homogeneous
generating syzygy of $S/I$ is $m+c$. Writing $-^{*}$ for $K$-duals, we obtain that
\begin{align*}
{\rm Tor}_{c}^S(S/I, K)_{m+c}
&\cong
((K\otimes_S {\rm Ext}^{c}_S(S/I, S))_{-m-c})^{*}\\
&\cong
({\rm Ext}^{c}_S(S/I, S)_{-m-c})^{*}\\
&\cong
({\rm Ext}^{c-1}_S(I, S)_{-m-c})^{*} \, .
\end{align*}
Since $S(-r-1)$ is the canonical module of $S$, local duality shows that the last of these is naturally isomorphic to
$H^{r-c+1}(\cI(m+c-r-1))$.
\end{proof}

If $X$ is a reduced curve in $\P^2_K$, then $X$ is arithmetically Gorenstein, and the
normalization of $X$ is Cohen-Macaulay. Moreover, there are many
Cohen-Macaulay rings between the homogeneous coordinate ring
of $X$ and its normalization. We can exploit this situation to improve
Theorem~\ref{Kloosterman}:

\begin{corollary}\label{plane curves}
If $X\subset \P^{2}_K$ is a reduced singular plane curve of degree $d$ over a
perfect field $K$, then
\[
\reg \C'_{X} =d-1-\inf\{m \mid h^{0}(\CO_{\overline X}(m))>{m+2\choose 2}\} \, ,
\]
and the minimal number of homogeneous generating
syzygies of $\C'_{X}$ of degree $d$ is one less than the number of components of $X$
over the algebraic closure of $K$.
Thus $\reg \C'_{X} < d-1$ if and only if $X$ is geometrically irreducible, and
in particular
$$
\reg (({\rm{Sing}} \, X)_{\red})\leq d-1 \,  ,
$$
with strict inequality when $X$ is geometrically irreducible.
\end{corollary}

In case $K$ is algebraically closed and $X$ is reducible, the result becomes attractively simple:

\begin{corollary}\label{plane curves 2}
Suppose that $K$ is algebraically closed and that $X\subset \P^2_{K}$ is a reduced but reducible curve.
Let
$I\subset S = k[x_0,x_1,x_2]$ be the saturated ideal of the subscheme of the conductor scheme
where at least two components meet.
If
$J$ is any unmixed ideal in $S$ with $\C'_{X}\subset J\subset I$, then $J$ has
regularity $d-1$ and
the minimal number of homogeneous generating syzygies of $J$ of degree $d$
is one less than the number of components of $X$.
\end{corollary}

In general the length $\delta$ of the conductor scheme of
a plane curve of degree $d$ ranges from
0 to ${d\choose 2}$---the latter being the case  when the curve is a union of lines. The regularity of a set of $\delta$ points in $\P^{2}_K$ ranges from $\delta$ down to
about $\sqrt{2\delta}$. Thus the statement that the regularity of the conductor is
bounded by $d-1$ (or $d-2$ in the case of geometrically irreducible curves) is quite strong.

Theorem \ref{main} says in particular that the reduced ideal of a set of singular
points on a reduced plane curve of degree $d$ has regularity at most $d-1$, and
regularity at most $d-2$ if the curve is geometrically reduced and irreducible.
Here is another version of this statement:

\begin{corollary}\label{indeg} If $I \subset K[x_0,x_1,x_2]$ is the reduced ideal of a finite
set of points in $\P^2_K$, then the symbolic square $I^{(2)}$ contains
no nonzero reduced form of degree $ \leq \reg I$ and no geometrically
reduced and irreducible form
of degree $\leq \reg I +1$.
\end{corollary}

This consequence of Corollary \ref{plane curves} seems to beg for generalization. What can one say for higher symbolic powers, or more variables? A famous conjecture of Nagata states that if $J$ is the defining ideal of $\delta$ \emph{general} points in $\P^{2}$ then the smallest degree of a form contained in $J^{(m)}$ is at least $m\sqrt \delta$ (see for example Harbourne~\cite{Harbourne} for a discussion and recent results). Our result for $m=2$ deals only with reduced forms, but with arbitrary sets of points.

\begin{example} Let $S = K[x_0,x_1,x_2]$ be the homogeneous coordinate ring of $\P^{2}_K$, and let $M$ be a generic map
$$
M: S(-7)\oplus S(-8)\longrightarrow S(-5)^{3}.
$$
Let $I$ be the ideal generated by
the three $2\times 2$ minors of $M$, which have degree 5.
From the formulas in Eisenbud~\cite{Eisenbud}, Chapter 3 one sees
that $I$ is the homogeneous ideal of a set $\Gamma$ of 19 reduced points.
Further, $M$ is the matrix of syzygies on the ideal $I$, and thus $\reg I = 7$.

Using Macaulay2 one can check that the smallest degree of a curve passing doubly through the points of $\Gamma$ is 10,
and that there is such a curve $X$ of degree 10 whose singularities consist of ordinary nodes at the 19 points of $I$. From Corollary \ref{plane curves} we see that $X$ is irreducible
and that the normalization $\overline X$ has
$h^{0}(\cO_{\overline X}(1)) = 3$; that is, $X$ is not the projection of a nondegenerate curve of
degree 10 in $\P^{3}_K$.
\end{example}

\smallskip

We can prove a weaker version of Theorem~\ref{main} without the Cohen-Macaulay assumption on $A$. Recall that a positively graded Noetherian $K$-algebra $A$ with homogeneous maximal ideal $A_{+}$ and graded canonical module $\omega_{A}$
is called \emph{quasi-Gorenstein} if $\omega_{A}\cong A(a)$ for some $a\in \ZZ$, called the $a$-invariant of $A$.
By local duality, $a = \reg H^{\dim A}_{A_{+}}(A)$, so, in case $A$ is generated in degree 1,
we have $a+\dim A\leq \reg A$, with equality when $A$ is Gorenstein.

\begin{theorem}\label{quasi-main} Let $X\subset \P^{r}_{K}$ be a geometrically reduced scheme
with homogeneous coordinate ring $A$.
Suppose that $B$ is a graded Cohen-Macaulay ring
with $A\varsubsetneq B\subset \overline A$.

If $A$ is quasi-Gorenstein,
then $\C_{B/A}$ is a Cohen-Macaulay $A$-module and an unmixed ideal in $A$ of codimension 1.
Moreover, $\reg \C'_{B/A} \leq \reg X$.
\end{theorem}

\smallskip

We can also say something about
the regularity of the ring $B$, considered as a graded module over $A$; see also
Ulrich and Vasconcelos [2004], the proof of 2.1(b).

\begin{proposition}\label{regB} Let $X\subset \P^{r}_{K}$ be a reduced scheme
with
homogeneous coordinate ring $A$.
Suppose that $B$ is a graded Cohen-Macaulay ring
with $A\varsubsetneq B\subset \overline A$.

If $A$ is Cohen-Macaulay, then the $A$-module $B/A$ is Cohen-Macaulay of codimension
$1$. Furthermore,
$$
\reg (B/A) \leq \reg X -2 \ \ {\rm and} \ \
\reg B \leq \reg X   - 1 \, .
$$
\end{proposition}

\medskip
\bigskip

We now proceed to the proofs.

\begin{proof}[Proof of Theorem~\ref{main}] To simplify the notation we write $\C'$ for $\C'_{B/A}$ and
$\C$ for $\C_{B/A}$. We have $\C \cong
{\rm{Hom}}_A(B,A)$. Since $B$ is a finite Cohen-Macaulay $A$-module and $\dim_A B = \dim A$,
we have
$\Ext^{1}_{A}(B,A) = 0$.  Therefore, applying the long exact sequence in $\Ext_A^{\bf{\cdot}}( -,A)$ to the  short exact sequence
$$
0\to A
\longrightarrow B
\longrightarrow B/A
\to 0 \, ,
$$
we obtain a homogenous isomorphism $A/\C \cong \Ext^{1}_{A}(B/A, A)$.
Since $A$ is Gorenstein, we have
$\omega_{A} = A(a)$.
Thus
$$A/\C
\cong\Ext^{1}_{A}(B/A, \omega_{A})(-a) \, .
$$
As $A$ is a Cohen-Macaulay $S$-module of codimension $c$,
local duality (or the change of rings spectral sequence) in turn gives
\begin{align*}\Ext^{1}_{A}(B/A, \omega_{A})
&\cong \Ext^{1+c}_{S}(B/A, \omega_{S})\\
&\cong \Ext^{1+c}_{S}(B/A, S)(-r-1) \, ,
\end{align*}
and we conclude that
\begin{align*}
A/\C
&\cong \Ext^{1+c}_{S}(B/A, S)(-a -r -1)
 \, .
\end{align*}

From the above short exact sequence it also follows that $B/A$ is a Cohen-Macaulay
$S$-module of codimension $c+1$ and hence has a minimal homogeneous free $S$-resolution
$F_{\mathbb{\cdot}}$ of length $c+1$. Now the isomorphism
$$S/\C'\cong A/\C \cong \Ext^{1+c}_{S}(B/A, S)(-a -r -1)$$
implies that dualizing $F_{\mathbb{\cdot}}$  into $S(-a-r-1)$ gives a minimal homogeneous
free $S$-resolution of $S/\C'$. In particular,
$\C'$ is a perfect ideal of codimension $c+1$ and
$$ \Tor^S_c(\C',K) \cong \Tor^S_{c+1}(S/\C',K) \cong
{\Hom}_{K}(B/A\otimes_SK,K)(-a -r -1) \, .$$
Since $A$ is Cohen-Macaulay, the degree shift can be rewritten as
\begin{align*} -a -r -1 &= -a  -  \dim A -c \\
&=-\reg A -c\\
&=-\reg X +1 -c \, .
\end{align*}
It follows that the two graded vector spaces
$$\Tor^S_c(\C',K)(\reg X -1 +c) \ \ {\rm and} \ \ \ (B/A)\otimes_AK $$
are dual to one another.

The maximal generator degree of the former controls the regularity of
$\C'$ because this ideal is perfect.
Therefore
\begin{align*}
\reg \C' &= {\rm{maxdeg}}\  \Tor^S_c(\C',K) - c \\
&= -\indeg(B/A) + \reg X - 1 \, ,
\end{align*}
as claimed. If $X$ is geometrically reduced and irreducible, then $A_0=K=(\overline A)_0$
and therefore $\indeg(B/A) >0$. The remaining assertions of the Theorem are now clear as
well.
\end{proof}

\begin{proof}[Proof of Corollary~\ref{plane curves}]
The last statements follow immediately from the formulas for
the regularity of $\C'_X$ and the number of its degree $d$ syzygies, so we prove those.

Let $A$ be the homogeneous coordinate ring of $X$.
For $m<d$ we have $\dim_K A_{m} = {m+2\choose 2}$, so the formula for
$\reg \C'_X$ follows at once from the one given in Theorem~\ref{main}.

If $B=\overline A$, then the number of homogeneous generating syzygies of
$\C'_X$ of degree $d$ is
${\rm dim}_K(B/A)_0$, again by Theorem~\ref{main}. On the other hand, $(B/A)_{0} = H^{0}(\pi_{*}\cO_{\overline X}/\cO_{X})$ and the dimension of this $K$-vector
space is one less than the number of components of $X$ over the
algebraic closure of $K$.
\end{proof}

\begin{proof}[Proof of Corollary~\ref{plane curves 2}]
Suppose that $X$ has $e$ irreducible components. Let $A$ be the homogeneous coordinate
ring of $X$ and let $B'$ be the direct product
of the homogeneous coordinate rings of the irreducible components of $X$,
a ``partial normalization'' of $A$. Notice that $I \subset \C'_{B'/A}$. By
Theorem~\ref{main} the ideal $\C'_{B'/A}$ has regularity
$d-1$ and its syzygy module has exactly $e-1$ homogeneous basis elements of degree $d$.
The same holds true for the conductor $\C'_X$.  Proposition~\ref{Cohen-Macaulay}
implies that the regularity and the number of syzygies of top degree are both
monotonic in the sequence of ideals $\C'_{X}\subset J\subset I \subset \C'_{B'/A}$, yielding the
desired formulas for the regularity and the number of degree $d$ syzygies of $J$.
\end{proof}

\begin{proof}[Proof of Corollary~\ref{improved Kloosterman}]
From Corollary \ref{plane curves} we know that $\reg I_{\Gamma} \leq d-1$ and
that this inequality is strict for irreducible curves.
If the curve is reducible and its components meet in ordinary nodes, then the part of the conductor
scheme concentrated at the points of intersection is the reduced scheme of the
intersection points itself, so we may apply Corollary~\ref{plane curves 2}.
\end{proof}

Before proving Theorem~\ref{quasi-main} we need a lemma on conductors, see also
Kunz [2005], 17.6.
Generalizing our previous notation, we write
$\C_{B/A} : = A:_AB$ for the conductor of any ring extension $A \subset B$.
Again, $\C_{B/A}$ is the unique largest $B$-ideal contained in $A$, and, whenever $A$
is Noetherian, this ideal contains a non zerodivisor of $B$ if and only
if the extension $A \subset B$ is finite and birational.

\begin{lemma}\label{extension} Let $A \subset B \subset C$ be extensions of rings.
If $\C_{B/A} = Bu$ for some non zerodivisor of $C$, then
$$ \C_{C/A} = \C_{C/B} \, \C_{B/A} \, .$$
\end{lemma}

\begin{proof} For any rings $A\subset B\subset C$ the inclusion
 $\C_{C/B} \, \C_{B/A}\subset \C_{C/A}$ follows from the definition
 of the conductor.
If $\C_{B/A} = Bu$, for some non zerodivisor $u\in C$,
then $\C_{C/A} \subset \C_{B/A}=Bu$ and we may write
$\C_{C/A} = \C u$ for some subset $\C$ of $B$. This subset is a
$C$-ideal because $\C_{C/A}$ is a $C$-ideal and $u$ is a non zerodivisor of $C$.
It follows that $\C \subset \C_{C/B}$ and therefore $\C_{C/A} \subset \C_{C/B} \, u
\subset \C_{C/B} \, \C_{B/A}$.
\end{proof}

\smallskip

\begin{proof}[Proof of Theorem~\ref{quasi-main}]
Since $X$ is geometrically reduced, we may
extend the ground field and assume that it is algebraically closed. We may then choose
a homogeneous Noether normalization inside $A$ over which the total ring of quotients of $A$ is separable,
and thus birational to a homogeneous hypersurface ring $A'\subset A$ that contains the Noether
normalization.

Write $\omega_{A} = A(a)$ and $\omega_{A'} = A'(a')$.
We claim that $\C_{A/A'}$ is generated by a homogeneous non zerodivisor $u$ on $B$ of degree $a'-a$.
This is because
$$
\C_{A/A'} \cong \Hom_{A'}(A,A') \cong  \Hom_{A'}(A, \omega_{A'}(-a')) \cong \omega_{A}(-a') \cong A(a-a').
$$
It follows from Lemma~\ref{extension} that $\C_{B/A'} = \C_{B/A} \, u$. By Theorem~\ref{main}
the $A'$-module $A'/\C_{B/A'}$ is Cohen-Macaulay of codimension $1$,
so $\C_{B/A'}$ is a maximal Cohen-Macaulay $A'$-module. Therefore
$\C_{B/A}$ is also a maximal Cohen-Macaulay module over $A'$ and thus over $A$.
This also shows that $\C_{B/A}$ is an unmixed ideal in $A$ of codimension 1,
because $\omega_A$ and hence $A$ satisfies Serre's condition $S_2$.

For the second statement, notice that
$$
\reg \C'_{B/A} = \reg(S/\C'_{B/A}) +1 = \reg(A/\C_{B/A}) +1 \, .
$$
The exact sequence
$$
0\to \C_{B/A}
\longrightarrow A
\longrightarrow A/\C_{B/A}
\to 0
$$
in turn shows that
$$
\reg(A/\C_{B/A}) \leq \max\{ \reg \C_{B/A} -1 , \, \reg A \} \, .
$$
Since $\C_{B/A'} = \C_{B/A} \, u$ with $u$ a homogeneous
non zerodivisor on $B$ of degree $a'-a$,
we see that
$
\reg \C_{B/A} = \reg \C_{B/A'}+a-a'
\, .
$
Combining this with the two displayed inequalities we obtain
$$
\reg \C'_{B/A} \leq \max \{ \reg \C_{B/A'} +a-a', \, \reg A +1 \} \, .
$$
Again, the exact sequence
$$
0\to \C_{B/A'}
\longrightarrow A'
\longrightarrow A'/\C_{B/A'}
\to 0
$$
gives
$$
\reg \C_{B/A'} \leq \max \{\reg(A'/\C_{B/A'}) +1, \,  \reg A' \} =\max \{\C'_{B/A'}, \,  a' + \dim A' \} \, .
$$
Finally, we apply Theorem~\ref{main} to the extension
$A'\subset B\subset \overline {A'} = \overline A$ and obtain
$$
\reg \C'_{B/A'} \leq a'+\dim A' \, .
$$
Combining the last three displayed inequalities we deduce
$$
\reg \C'_{B/A} \leq  \max\{a+ \dim A, \, \reg A +1 \} = \reg A + 1 = \reg X \, ,
$$
as desired.
\end{proof}

\smallskip

\begin{proof}[Proof of Proposition~\ref{regB}]
One sees immediately, as in the proof of
Theorem~\ref{main}, that $B/A$ is a Cohen-Macaulay $A$-module of codimension $1$.
As for the other claims,
we may assume that the ground field $K$
is infinite. In this case $A$ admits a homogeneous Noether normalization $A'$.
Since the Cohen-Macaulay rings $A$ and $B$ are finitely
generated graded modules over the polynomial ring $A'$, it follows that they
are maximal Cohen-Macaulay modules and hence free. Thus
$$ 0 \to A \longrightarrow B $$
is a homogeneous free resolution of minimal length of the $A'$-module $B/A$. As the latter module is Cohen-Macaulay,
its regularity can be read from the last module in
the minimal homogeneous free resolution. It follows that $\reg _{A'}B/A  \leq \reg _{A'}A -1$.
Since $A$ is finite over $A'$, the regularity of an $A$-module is the same as
its regularity as an $A'$-module. Therefore $\reg_A B/A \leq \reg A -1$, which also
implies $\reg _AB \leq \reg A$.
\end{proof}

\section{Completeness of the Adjoint Series}\label{adjoint section}

If $X$ is irreducible, then Theorem~\ref{Kloosterman} can be deduced from a
classical result known as the ``completeness of the adjoint series''. We will explain how below, but first we explain the meaning of the terms.

An \emph{adjoint of degree $e$} is a form of degree $e$ satisfying an ``adjoint condition'' at each singularity. For example, at an ordinary node or cusp the adjoint condition is simply to vanish at the point; if the singularity is an ordinary $k$-fold point, then the
adjoint condition is vanishing to order $k-1$. In general  the adjoint condition is to
be contained in the local conductor ideal at the point. Thus, from our point of view, an adjoint is simply a form of degree $e$ contained in $\C_{X}$.

The adjoint conditions at the singular points of $X$ give rise to a divisor $D$ on $\overline X$, namely the zero locus of the pull-backs of all the (local) functions satisfying the adjoint conditions. The \emph{completeness of the adjoint series} is the statement that
the natural inclusion of $(\C_{X})_{e}$, the space of adjoints of degree $e$, to
$H^{0}(\cO_{\overline X}(e)(-D))$, is an isomorphism. From our point of view this is just
the statement that $\C_{X}$ is an ideal of both $A$, the coordinate ring of $X$, and
$B = \oplus_{e}H^{0}(\cO_{\overline X}(e))$, so that the natural inclusion
$\C \subset \C B$ is in fact an equality.

We can now deduce the case of Theorem~\ref{Kloosterman} when the curve $X$ is irreducible. First, it suffices to show that the set $\Gamma$ of all singular points of $X$ has regularity $\leq d-2$, since then any subset has regularity $\leq d-2$ as well. To do this we must show
that $H^{1}(\cI_{\Gamma}(d-3)) = 0$.
So it suffices to show that
the nodes and cusps impose independent conditions on the forms of degree $d-3$.

Under the hypothesis of the Theorem that $X$ has only ordinary nodes and cusps, $\cI_{\Gamma} = \widetilde{\C'_{X}}$.
In this case, the ``completeness of the adjoint series'' says that the vector space of
forms of degree $d-3$ vanishing on $\Gamma$ gives the complete canonical series on
$\overline X$, and thus has dimension genus$(X)$, which is
$$
{d-1\choose 2} - \deg \Gamma = \dim S_{d-3} - \deg \Gamma;
$$
that is, the conditions imposed by $\Gamma$ are independent, as required.

\section{Syzygies of an Almost Complete Intersection}

In the case where $X$ is a plane curve with only ordinary nodes as singularities, the conductor ideal
is the unmixed part of the Jacobian ideal, the ideal of partial derivatives of the defining
equation of $X$, and we can derive a surprising consequence for the syzygies of
these partial derivatives:

\begin{corollary}\label{derivatives} Let $X\subset \P^{2}_K$ be a reduced curve of degree $d$ over an algebraically closed field whose characteristic is not a divisor of $d$. Write $S=K[x_0,x_1,x_2]$
 for the homogeneous coordinate ring of $\P^2_K$. Assume that
$F(x_{0}, x_{1},x_{2})=0$ is the defining equation of $X$ and
that $X$ has only ordinary nodes as singularities. If
$$
\sum_{i=0}^{2}G_{i}\frac{\partial F}{\partial x_{i}} = 0
$$
is a homogeneous relation in $S$,
then $\deg G_{i}\geq d-2$ whenever $G_i \neq 0$, and equality is possible if and only if $X$ is reducible.
\end{corollary}

This result was proven independently, in characteristic zero, in Dimca and Sticlaru~\cite{Dimca-Sticlaru} (Theorem 4.1), using Hodge theory. We deduce it from a much more general statement about almost complete intersections given in Proposition~\ref{reg from fitting}.


\begin{proposition}\label{reg from fitting}
Let $S = K[x_{0},\dots, x_{r}]$ be a polynomial ring
 over a field and let $I\subset S$ be a homogeneous perfect
ideal of codimension $g$.
Let $f_{1},\dots,f_{g+1}$ be homogeneous elements of $I$ such that
$f_{1},\dots,f_{g}$ form a regular sequence. Set $\a := (f_{1}, \dots,f_{g})\subset J := (f_{1},\dots,f_{g+1})$. If
$I$ is the unmixed part of $J$, meaning the intersection of the primary components of $J$
having codimension $g$,
then
$$
\omega_{S/I} = ((\a:f_{g+1})/\a)(-r-1+\sum_{i=1}^{g}\deg f_{i}) \, .
$$
Furthermore,
$$
\reg S/I = \reg S/\a -\indeg ((\a:f_{g+1})/\a)\, ,
$$
and if $f_{g+1}$ has maximal degree among the $f_{i}$,
this can be written as
$$
\reg S/I = \reg S/\a -\indeg (I_{1}(\phi)/\a) \, ,
$$
where $\phi$ is a matrix of homogeneous syzygies of $f_1, \ldots,f_{g+1}$ and $I_{1}(\phi)$ denotes
the ideal generated by the entries of $\phi$.
\end{proposition}

\begin{proof} We have
$$\reg S/I = -\indeg\ \omega_{S/I} + \dim S/I \, $$
since $S/I$ is Cohen-Macaulay, and by linkage
$$
\omega_{S/I} = ((\a:I)/\a)(-r-1+\sum_{i=1}^{g}\deg f_{i}) \, .
$$
Thus, using the fact that $\reg S/\a = -g+\sum_{i=1}^{g}\deg f_{i}$, we obtain
$$
\reg S/I = \reg S/\a -\indeg ((\a:I)/\a) \, .
$$
Since $\a$ is unmixed,
$$
\a:I = \a:f_{g+1} \, ,
$$
proving the first two assertions.

If $f_{g+1}$ has maximal degree among the $f_{i}$, then
the non-zero elements in the last row of the syzygy matrix $\phi$, which are the generators
of $\a:f_{g+1}$, have the lowest degree among the non-zero entries of
their respective column. The columns with zero last entry, on the other hand, are
relations on the regular sequence $f_{1}, \dots, f_{g}$ and hence have entries in $\a$. It
follows that $$\indeg ((\a:f_{g+1})/\a) = \indeg(I_1(\phi)/\a) \, .$$
\end{proof}

We are now ready to prove Corollary \ref{derivatives}.

\begin{proof} Our assumption on the characteristic implies that $F$ is contained in the
ideal of $S$ generated by $\frac{\partial F}{\partial x_0}, \frac{\partial F}{\partial x_1}, \frac{\partial F}{\partial x_2}$. If $X$ is smooth, this ideal has codimension
at least 3 and $X$ is irreducible. Otherwise we apply Proposition \ref{reg from fitting} with $r=2$, $I=\C'_X$, and
$J=(\frac{\partial F}{\partial x_0}, \frac{\partial F}{\partial x_1}, \frac{\partial F}{\partial x_2})$.
One has $J \subset I$, and equality holds locally at every prime ideal of codimension
two because all singularities of $X$ are ordinary nodes. Since $I$ is unmixed,
we conclude that $I$ is the unmixed part of $J$. After a linear change of variables we may
assume that $\frac{\partial F}{\partial x_0}, \frac{\partial F}{\partial x_1}$ is a regular sequence of
two forms of degree $d-1$. Thus the Proposition
shows that
$$
\reg S/I = 2(d-2) -\indeg (I_{1}(\phi)/\a) \, .
$$
Finally, according to Corollary \ref{plane curves} one has $\reg S/I \leq d-2$, and equality
holds if and only if $X$ is reducible.
\end{proof}


\section{The Conductor of a Reducible Hypersurface}

The estimates obtained in Section 3
point to a difference between the reducible and the irreducible case. In the
present section we study this phenomenon by relating the conductor ideal
of a reducible subscheme to the conductor ideals of its irreducible
components.
This is done
in the next proposition. Our proof is an adaptation of arguments in Kunz [2005],
17.6, 17.11, 17.12, where the case of local rings of plane curve
singularities is treated.

\smallskip

\begin{proposition}\label{reduced} If
$X\subset \P^{r}_{K}$ is a reduced hypersurface with distinct irreducible components
$X_i$, then
$$
\C'_{X} = \sum_{i}\C'_{X_{i}}\prod_{j\neq i} I_{X_{j}} \, .
$$
\end{proposition}
\begin{proof}
We write $F$, $F_i$ for the defining homogeneous polynomials
of $X$, $X_i$ in the polynomial ring $S: = K[x_0, \ldots, x_r]$,
$G_{i} := \prod_{j\neq i} F_{j}$, and $A$, $A_i$
for the homogeneous coordinate rings $S/(F)$, $S/(F_i)$. Consider the ring
extensions
$$
A \, \subset \, B:= {\times}_i A_i \, \subset \, C:= \overline A = \overline B = {\times}_i \overline{A_i} \, .
$$
Further, write $e_i$ for the $i$th idempotent of $B$ and notice that $A_ie_i = Ae_i =Se_i$,
$B = \sum_i Se_i$, $1_A= \sum_i e_i$.

Clearly $\C_{C/B} = {\times}_i \C_{\overline{A_i} / A_i} = \sum _i \C'_{X_i}e_i$.

Next, we claim that $\C_{B/A} =B( {\sum}_i G_ie_i)$. Indeed, the image of an element
$H \in S$ in $A$ belongs to $\C_{B/A}$ if and only if $SHe_i \subset A$ for every
$i$, which means that for every $i$ there exists $H_i \in S$ so that
$H_i \equiv H \, {\rm{mod}} \, F_i$ and  $H_i \equiv 0 \, {\rm{mod}} \, G_i$. This is equivalent
to $He_i \in SG_i e_i$. Hence indeed $\C_{B/A} ={\sum}_i SG_ie_i =B( {\sum}_i G_ie_i)$, where the last equality holds because the $e_i$ are orthogonal idempotents.

Since the $B$-ideal $\C_{B/A}$ is generated by a single non zerodivisor of $C$, it follows that
$ \C_{C/A} = \C_{C/B} \, \C_{B/A}$ according to Lemma \ref{extension}. Therefore
\begin{align*}
\C'_X  \, 1_A = \C_{C/A} &= ( {\sum}_i G_ie_i)(\sum _i \C'_{X_i}e_i) \\
&= \sum_i G_i\C'_{X_i} e_i \\
&=(\sum_i G_i\C'_{X_i}) (\sum_i e_i) \\ &= (\sum_i G_i\C'_{X_i}) \,  1_A \, .
\end{align*}
We deduce that $\C'_X= \sum_i G_i \C'_{X_i}$ because the ideals on both sides
of the equation contain $F$.
\end{proof}

\smallskip

\begin{corollary}\label{sequence}
In the setting of Proposition $\ref{reduced}$ we write $S:=K[x_0, \ldots, x_r]$ and
$d:={\rm deg} \, X$, $d_i:={\rm deg} \, X_i$. For every $i$ there is an exact sequence
$$
0 \rightarrow S(-d) \longrightarrow \C'_{\cup_{j\neq i}X_j} (-d_i) \oplus \C'_{X_i}(-d+d_i) \longrightarrow \C'_X \rightarrow 0 \, .
$$
In particular,
$$
\reg  \C'_X \leq {\rm max}_i \{ d-1, \reg \C'_{X_i}+d-d_i \} .
$$
\end{corollary}
\begin{proof} We write $F=F_{i}G_{i}$ as in the previous proof. We will construct the
desired sequence in the more precise form
$$
0 \rightarrow SF \longrightarrow F_i \C'_{\cup_{j\neq i}X_j}  \oplus  G_i \C'_{X_i} \longrightarrow \C'_X \rightarrow 0 \,.
$$
Proposition~\ref{reduced} gives $F_i \C'_{\cup_{j\neq i}X_j} + G_i \C'_{X_i} =\C'_X$,
so we may take the right hand map in the desired sequence to be addition,
and it suffices to show that
$$
F_i \C'_{\cup_{j\neq i}X_j} \cap G_i \C'_{X_i} =SF.
$$
The right hand side
of this equation is  contained in the left hand side because
$\C'_{\cup_{j\neq i}X_j}\ni G_{i}$
and $\C'_{X_i}\ni F_{i}$.
The reverse inequality holds because $F_{i}$ and $G_{i}$ are relatively prime,
and thus
$$
F_i \C'_{\cup_{j\neq i}X_j} \cap G_i \C'_{X_i}\subset
F_i S \cap G_i S = F_{i}G_{i}S = FS \, .
$$

The assertion about the regularity of $\C'_X$ follows from the exact sequence, using induction on the number of irreducible components of $X$.
\end{proof}

\smallskip

\begin{corollary}\label{curves} Let
$X\subset \P^{2}_{K}$ be a reduced plane curve over an algebraically closed field with
distinct irreducible components
$X_i$ and write ${\mathcal J}_{X_i}$ for the saturated ideals defining the
reduced singular sets $({\rm{Sing}} \, X_i)_{\rm red} \subset \P^2_K$.
If $X$ has only ordinary nodes and cusps as singularities, then the ideal
$$
\sum_{i}{\mathcal J}_{X_{i}}\prod_{j\neq i} I_{X_{j}}
$$
is the saturated ideal defining the reduced singular set ${(\rm{Sing}} \, X)_{\rm red} \subset \P^2_K$.
\end{corollary}
\begin{proof}
Because of the assumption on the singularities, the conductor ideals $\C'_X$
and $\C'_{X_i}$ are reduced. Now apply Proposition \ref{reduced}.
\end{proof}


\section{Examples of Resolutions of Singular Sets}

We exhibit two situations in which we can specify the resolution of the singular ideal of a plane curve completely. The third example
shows that the bound coming from B\'ezout's Theorem can be much less sharp than that of Corollary~\ref{plane curves}. In this section $X \subset \P^2_K$ will always be a reduced curve over
an algebraically closed field $K$.

\begin{example}
 Suppose that $X\subset \P^{2}_{K}$, given by $F=0$, is the union of more than 1 distinct smooth components $X_{i}$ meeting transversely; in other words, $X$ has smooth irreducible components
 and only ordinary nodes as singularities. Suppose that the equation of $X_{i}$ is $F_{i}=0$, so that $F = \prod_{i=1}^{\ell}F_{i}$. According to Corollary \ref{curves} for instance, the reduced singular set of $X$  has homogeneous ideal generated by the $\ell$ products
$$
G_{i}: = \prod_{j\neq i} F_{j} \, .
$$
It follows, in particular,
that the ideal $I$ generated by the $G_{i}$ has codimension 2.

Hence this ideal is the
flat specialization of the perfect ideal
$$
(\{ g_{i}:=\prod_{j\neq i} x_{j} \ | \ 1 \leq i \leq \ell\}) \subset K[x_{1}, \dots, x_{\ell}] \, .
$$
Since all the products $x_{i}g_{i}$ are equal to $\prod_{j=1}^{\ell}x_{j}$, we have $\ell-1$
syzygies of the form
$x_{i+1}g_{i+1} - x_{i}g_{i}$ on the $g_{i}$, and these generate all the syzygies.
Thus $I$ is also perfect and its
homogeneous minimal free resolution over the polynomial ring $S$ in 3 variables has the form
$$
0\to \oplus_{i=1}^{\ell-1}S(-d) \longrightarrow \oplus_{i=1}^{\ell}S(-d+d_{i}) \longrightarrow	 I \to  0 \, ,
$$
where $d=\deg X$ and $d_{i} = \deg X_{i}$. In particular, we see directly that
$\reg I = d-1$ as claimed.
\end{example}

\smallskip

\begin{example}
 Let $X\subset \P^2_K$ be a rational curve of degree $d \geq 2$ with only ordinary nodes and cusps as singularities, so that $({\rm{Sing}}\, X)_{\rm{red}}$ consists of ${d-1\choose 2}$ distinct points. We claim that the minimal free resolution of the saturated ideal $I:=I_{({\rm{Sing}}\, X)_{\rm{red}}}$ has the form
$$
0\to S(-d+1)^{d-2} \longrightarrow S(-d+2)^{d-1} \longrightarrow	I\to 0 \, ,
$$
and thus $\reg I= d-2$. As $\reg \C'_X \geq \reg I$ and $X$ is irreducible, it also follows that
$\reg \C'_X=d-2$ and $h^{0}(\CO_{\overline X}(1))>3$, according to Corollary \ref{plane curves}.

To see the claim about the resolution, we first note that no curve of degree $d-3$ can pass through all the nodes of
$X$. This is because such a curve would meet $X$ in a scheme of degree
at least $2{d-1\choose 2} = (d-1)(d-2)$, whereas by B\'ezout's Theorem the intersection
scheme could only be of length $d(d-3) < (d-1)(d-2)$.

Because $I$ is the ideal of a reduced set of points, any linear form $x$
that  vanishes at none of the points is  a non zerodivisor modulo $I$.
If we reduce $I$ modulo $x$ we get a homogeneous ideal of finite colength ${d-1\choose 2}$,
not containing any form of degree $d-3$,
in the polynomial ring in 2 variables. The only such ideal is the $(d-2)$-nd power of the
maximal homogeneous ideal, and this has minimal free resolution as above. Since reducing modulo a linear non zerodivisor preserves the
shape of the resolution, we are done.
\end{example}

\smallskip

\begin{example}
In some cases, the regularity bound given in
Corollary~\ref{plane curves} can be deduced simply from
B\'ezout's Theorem. For example, suppose that $X \subset \P^2_K$ is an irreducible curve of degree $d$ and the singular set of $X$ is the transverse complete intersection of curves $E$ and $F$ of degrees $e < f$, respectively. In this case Corollary \ref{plane curves} asserts in particular that
$\reg (E\cap F) = e+f-1 \leq d-2$, that is, $e+f < d$.
By B\'ezout's Theorem, the degree of $E\cap X$ is $de$. But $E$ meets $X$ with multiplicity
at least 2 at each of the $ef$ points of $E\cap F$, so $2ef\leq de$, or $e+f < 2f\leq d$ as claimed.

On the other hand, suppose $I \subset K[x_0,x_1,x_2]$ is the ideal generated
by the $m\times m$ minors of
a generic $m+1\times m$ matrix $M$ whose first column consists of generic forms of degree $2m-1$
and whose other entries are generically chosen quadratic forms.
From the formulas in Eisenbud \cite{Eisenbud}, Chapter 3 we see that
\begin{itemize}
\item $I$ is generated by $m+1$ forms of degree $4m-3$;
\item $I$ is the ideal of a set $\Delta$ consisting of
$$
\delta = 20{m-1\choose 2}+18m - 17
$$
reduced points;
\item $\reg I = 6m-5$.
\end{itemize}
If $X$ is an irreducible curve singular at all the points of $\Delta$, then
we can find a linear combination of the $m+1$ generators of $I$ defining a curve
that meets $X$ in a scheme of length at least $m+ 2\delta$, so B\'ezout's Theorem shows that the degree $d$
of any such curve $X$ satisfies
$$
d\geq \frac{m+2\delta}{4m-3} \, ,
$$
which, after substituting the value of $\delta$, becomes $d\geq 5m-2$.
However, Corollary~\ref{plane curves} shows that in fact we must have
$d\geq \reg I+2 = 6m-3$. (In experiments, the minimal degree in these
circumstances---given that the matrix $M$ is chosen generally---for
the first values of $m\geq 2$ seems actually to be equal to $8m-6$.)
\end{example}

\section{Other results on the situation of the nodes}\label{other results}
Coolidge~\cite{Coolidge} says that ``One of the most important unsolved probems in the whole theory of plane curves [is] the situation of the permissible singular points.''
But we know of very few results shedding light on this problem. In fact, other than the results of Kloosterman and of this paper the only references of which we are aware are:
\begin{itemize}

\item On pp 389 ff Coolidge gives some results for rational curves of degrees 6 and 7 in $\P^{2}_K$.

\item A result of Arbarello and Cornalba~\cite{Arbarello} shows that vanishing doubly at $
\delta$ nodes imposes independent conditions on forms of degree $d$ whenever
$
{d+2\choose 2} \geq 3\delta.
$
\end{itemize}

\bigskip

\vbox{\noindent Author Addresses:\par
\smallskip
\noindent{David Eisenbud}\par
\noindent{Department of Mathematics, University of California, Berkeley,
Berkeley, CA 94720}\par
\noindent{eisenbud@math.berkeley.edu}\par
\smallskip
\noindent{Bernd Ulrich}\par
\noindent{Department of Mathematics, Purdue University, West Lafayette, IN 47907}\par
\noindent{ulrich@math.purdue.edu}\par
}

\end{document}